\documentclass[a4paper,12pt]{article}

\usepackage{amsmath, amsthm, amssymb}

\DeclareMathOperator{\argmax}{arg\ max}
\usepackage{enumerate}
\usepackage{graphicx}
\usepackage[numbers]{natbib}
\usepackage[top=1in, bottom=1in, left=1in, right=1in]{geometry}

\title{Multivariate Laplace's approximation with estimated error and application to limit theorems}
\author{Tomasz M. \L api\'nski \small{(email: 84tomek@gmail.com})\\ \\Gda\'nsk University of Technology,
\\{\small ul. Gabriela Narutowicza 11/12, 80-233 Gda\'nsk, Poland}}

\begin{document}

\newtheorem{remark}{Remark}
\newtheorem{definition}{Definition}
\newtheorem{proposition}{Proposition}
\newtheorem{theorem}{Theorem}
\newtheorem{lemma}{Lemma}
\newtheorem{proofoflemma}{Proof of Lemma}
\newtheorem{proofofproposition}{Proof of Proposition}

\maketitle

\begin{abstract}
	In this paper we obtain an approximation for the multivariate Laplace's integral with a large parameter and estimate error term for two cases, when the maximum of the exponent is in the interior of the domain and on the boundary. We are specifically interested in the situation when the function in the exponent depends on the large parameter. As an application we prove weak law of large numbers and central limit theorem. The second result gives different limiting distributions for two cases mentioned above. When the maximum of the exponent is in the interior of the domain it is Normal distribution and if it is on the boundary, it is Exponential in one direction of integration and Normal in other directions.
\end{abstract}


\begin{msc}
	41A60, 41A63, 60F05
\end{msc}
\\
\\
\copyright 2019. This manuscript version is made available under the CC-BY-NC-ND 4.0 license http://creativecommons.org/licenses/by-nc-nd/4.0/.\\
\\
https://doi.org/10.1016/j.jat.2019.105305
\\

\maketitle


\section{Introduction}

	\qquad We start with a result for one dimensional Laplace's integral and the maximum of the exponent on the boundary. This univariate case is well studied and error estimates are known, see e.g. \cite{Laplace_method_univariate_coefficients_1}, \cite{Laplace_method_univariate_coefficients_2}, \cite{error_bounds_for_univariate_case_book}. However, we need a specific estimates for further development.\\
	\indent Our main concern is with finite dimensional integrals. Without the rate of convergence and explicit error bounds it is proved in \cite{Laplace_approximation_multivariate_boundary_case_bleistein}, \cite{Laplace_approximation_multivariate_boundary_case_breitung}, \cite{Laplace_approximation_multivariate_boundary_case_wong}.\\
	\indent Let us consider an open set $\Omega\subset\mathbb{R}^{m}$ and define the integral
	\begin{equation}
		\label{Laplace_integral_interior}
		I(N):=\int_{\Omega}g(x)e^{Nf(x,N)}dx,
	\end{equation}
	where functions $f$ and $g$ are sufficiently regular, $N\in\mathbb{Z}_{+}$ and $f(\cdot,N)$ has as a unique maximum in the interior of $\Omega$. 
	Without dependence of $f$ on $N$ a precise error estimate for the above integral is given in \cite{Laplace_method_approach_kolokoltsov}. Then in \cite{error_bounds_for_interior_case_majerski} the authors provide an explicit error bound with slightly different underlying assumptions and various rate of convergence. Here we developed a method of the approximation and remainder estimate of (\ref{Laplace_integral_interior}) in the spirit of \cite{Laplace_method_approach_kolokoltsov}.\\
	\indent Our main result is an approximation of the integral 
	\begin{equation}
		\label{Laplace_integral_boundary}
		I(N):=\int_{\Omega\cap \{x:x_{1}\geq 0\}}g(x)e^{Nf(x,N)}dx,
	\end{equation}
	with $\Omega$, $N$, $g$ as in (\ref{Laplace_integral_interior}) and $f$ also sufficiently regular but with unique non-critical maximum on the boundary $\{x:x_{1}=0\}$. Our result includes the rate of convergence and explicit remainder estimate.\\
	\indent For the application, we consider probability distribution function where the probability of an event is constructed by taking the Laplace's integral over the event set and normalizing with the integral over the domain. We also consider a random vector whose values are the points of the domain of the integral. The first limit theorem, law of large numbers, gives us its expectation in the limit of $N$. It is equal to the point of maximum of the function in the exponent of Laplace's integral. The second result, central limit theorem, gives the distributions of the fluctuations. They are different for two cases of maximum. When it is in the interior of the domain it is Normal, and when it is on the boundary it is Exponential in one direction and Normal in other directions. The law of large numbers in this context is well understood in the theory of Large Deviation see e.g. \cite{Probability_kallenberg}, so our main contribution are the precise rate of convergence under certain regularity assumptions on $f$ and $g$.\\
	\indent Another application of the Laplace's approximation in Probability Theory was developed in \cite{application_of_Laplace_to_random_chaos_hashrova}, where authors study the asymptotics of the distribution of Gaussian and Weibullian random chaoses.\\
	\indent The motivation for the development of the results in this paper is to provide a framework and methodology to prove analogous results for sums instead of integral in (\ref{Laplace_integral_interior}) and (\ref{Laplace_integral_boundary}). This is done in \cite{lapinski}. Then the results from \cite{lapinski} are used to prove general limit theorems in the context of Statistical Mechanics in \cite{lapinski_theorem}.\\
	\indent All the results are extendable to have a positive real large parameter $N$ instead of the integer. Integrals (\ref{Laplace_integral_interior}) and (\ref{Laplace_integral_boundary}) with $N=1/\hbar$, where $\hbar$ is the Planck constant, are often met in quantum physics.


\section{Laplace's approximation}
	\qquad We consider an open set $\Omega\subset\mathbb{R}^{m}$ and a closed ball $\Omega'\subset\Omega$ with the center at the origin, radius $\varepsilon$ and volume $|\Omega'|$. Then we introduce a function $f:\Omega\times\mathbb{Z}_{+} \to \mathbb{R}$ for which derivatives up to third order exists on $\Omega'$ and are uniformly bounded. Further, its Hessian matrix is nonsingular. For all $N\geq N_{0}$, some $N_{0}\in\mathbb{Z}_{+}$, function $f(\cdot,N)$ have a unique maximum at $x^{*}(N)\in\Omega'$ such that
	\begin{equation}
		\label{Laplace_integral_delta}
		\Delta:=\inf_{N\geq N_{0}, x\in\Omega\backslash\Omega'}\{f(x^{*}(N),N)-f(x,N)\}>0.
	\end{equation}
	We choose the origin of our coordinate system to be the point $x^{*}=\lim_{N\to\infty}x^{*}(N)$.\\
	We also consider a $C^{1}(\Omega')$ function $g:\Omega\to\mathbb{R}$ and define constants
	\begin{align}
		\label{Laplace_integral_constants_G_G1}
		&G:=\sup_{x\in\Omega'}\|g(x)\|<\infty,\quad G^{(1)}:=\sup_{x\in\Omega'}\|Dg(x)\|<\infty,\\
		\label{Laplace_integral_bound}
		&C>0,\ C\geq\int_{\Omega}|g(x)|e^{N_{0} f(x,N)}dx,\ \text{for all}\ N\geq N_{0}.
	\end{align}
	Let us assume integrals (\ref{Laplace_integral_interior}) and (\ref{Laplace_integral_boundary}) exists and are finite. Further,
	\begin{enumerate}[(a)]
		\item for the integral (\ref{Laplace_integral_interior}) we assume $f(\cdot,N)$ has a maximum in the interior of $\Omega'$ and we introduce a constants
			\begin{align}
				\label{Laplace_integral_interior_constant_F'2}
				&F'^{(2)}:=\inf_{x\in\Omega', N\geq N_{0}}\|D^{2}f(x,N)^{-\frac{1}{2}}\|^{-2}>0,\\
				\label{Laplace_integral_interior_constant_F'2det}
				&F'^{(2)}_{det}:=\inf_{x\in\Omega', N\geq N_{0}}\sqrt{|\det D^{2}f(x,N)|}>0,\\
				\label{Laplace_integral_interior_constant_F2}
				&F^{(2)}:=\sup_{x\in\Omega',N\geq N_{0}}\|D^{2}f(x,N)\|<\infty,\\
				\label{Laplace_integral_interior_constant_F3}
				&F^{(3)}:=\sup_{x\in\Omega',N\geq N_{0}}\|D^{3}f(x,N)\|<\infty,
			\end{align}
		\item for the integral (\ref{Laplace_integral_boundary}) we assume $f(\cdot,N)$ has maximum on the boundary, i.e. $x^{*}(N)\in\{x:x_{1}=0\}$ and also introduce a constants
			\begin{align}
				\label{Laplace_integral_boundary_constant_F'1}
				&F'^{(1)}:=\inf_{x\in\Omega', N\geq N_{0}}\bigg|\frac{\partial f(x,N)}{\partial x_{1}}\bigg|>0,\\
				\label{Laplace_integral_boundary_constant_F'2}
				&F'^{(2)}:=\inf_{x\in\Omega', N\geq N_{0}}\|D_{y}^{2}f(x,N)^{-\frac{1}{2}}\|^{-2}>0,\\
				\label{Laplace_integral_boundary_constant_F'2det}
				&F'^{(2)}_{det}:=\inf_{x\in\Omega', N\geq N_{0}}\sqrt{\det|D_{y}^{2}f(x,N)|}>0,\\
				\label{Laplace_integral_boundary_constant_F2}
				&F^{(2)}:=\sup_{x\in\Omega',N\geq N_{0}}\|D^{2}f(x,N)\|<\infty,\\
				\label{Laplace_integral_boundary_constant_F3}
				&F^{(3)}:=\sup_{x\in\Omega',N\geq N_{0}}\|D^{3}f(x,N)\|<\infty,
			\end{align}
			where $y=(x_{2},\ldots,x_{m})$ and $D_{y}$ is a differential operator in that coordinates.\\
			Furthermore, we assume $\argmax_{y\in\Omega'(x_{1})} f(x_{1},y,N)$ is a single point for any $x_{1}\in(0,\varepsilon)$, where $\Omega'(x_{1})=\{y:(x_{1},y)\in\Omega\}.$
	\end{enumerate}
	
	\begin{remark}
		The situation when the boundary of the domain of integration in (\ref{Laplace_integral_boundary}) is an arbitrary curve, smooth in the neighborhood of the origin, can be reduced to the case with the boundary $\{x:x_{1}=0\}$ by appropriate local change of the coordinates. 
	\end{remark}


	\subsection{Univariate integral}

	\begin{theorem}
		For the integral (\ref{Laplace_integral_boundary}) with $\Omega=[0,\infty)$ following approximation holds
		\begin{align*}
			\int_{\Omega}g(x)e^{N f(x,N)}dx=e^{N f(0,N)}\frac{1}{N}\bigg(\frac{g(0)}{|f'(0,N)|}+\frac{\omega(N)}{N}\bigg),
		\end{align*}
		where $N\geq N_{1}$ for $N_{1}=\max\{\lceil{1}/{\varepsilon^{2}}\rceil,N_{0}\}$, $\omega(N)=O(1)\ \text{as}\ N\to\infty$ and
		\begin{align*}
			|\omega(N)|\leq&\frac{GF^{(2)}}{\big(F'^{(1)}\big)^{3}}\exp\bigg(\frac{1}{2}F^{(2)}\bigg)+\frac{G^{(1)}}{\big(F'^{(1)}\big)^{2}}+\frac{2G}{F'^{(1)}}N\exp\big(-N^{1/2}F'^{(1)}\big)+\\
				&+N^{2}\exp\big(-N\Delta-N_{0}(f(0,N)-\Delta)\big)C.
		\end{align*}
	\end{theorem}

	\begin{proof}
		Let us define
		\begin{equation*}
			\label{Extended_Laplace_proof_laplace_approximation_one_dimension_interior_of_domain_integral_gaussian}
		           I_{B}(N):=g(0)e^{N f(0,N)}\frac{1}{N}\frac{1}{|f'(0,N)|}=g(0)e^{N f(0,N)}\int_{0}^{\infty}e^{-N|f'(0,N)|x}dx,
		\end{equation*}
		and introduce a set $U_{N}:=\{x:|x|\leq1/N^{1/2}\}$. For all $N\geq N_{1}$ with $N_{1}=\max\{\lceil1/\varepsilon^{2}\rceil,N_{0}\}$ we have that $U_{N}\subset \Omega'$. Then with use of Taylor's Theorem we decompose $I(N)$ and $I_{B}(N)$
		\begin{align*}
			\label{Extended_Laplace_proof_laplace_approximation_one_dimension_interior_of_domain_integral_gaussian_decomposition}
			I(N)=&I_{11}(N)+I_{12}(N)+I_{2}(N)+I_{3}(N):=g(0)\int_{U_{N}}e^{N f(x,N)}dx+\int_{U_{N}}g'(x_{\theta})xe^{N f(x,N)}dx+\\
			&+\int_{\Omega'\backslash U_{N}}g(x)e^{N f(x,N)}dx+\int_{\Omega\backslash\Omega'}g(x)e^{Nf(x,N)}dx,\\
			I_{B}(N)&=I_{B1}(N)+I_{B2}(N):=g(0)e^{N f(0,N)}\int_{U_{N}}e^{-N|f'(0,N)|x}dx+\\
			&+g(0)e^{N f(0,N)}\int_{\Omega\backslash U_{N}}e^{-N|f'(0,N)|x}dx.
		\end{align*} 
		Here and everywhere in the proofs $x_{\theta}$ denotes a point between $x$ and maximum, which might be different in different instances.\\ 
		Now, we put together above decompositions
		\begin{equation*}
			|I(N)-I_{B}(N)|\leq|I_{11}(N)-I_{B1}(N)|+|I_{12}(N)|+|I_{2}(N)|+|I_{3}(N)|+|I_{B2}(N)|,
		\end{equation*}
		and approximate each term.\\
		To approximate the expression $|I_{11}(N)-I_{B1}(N)|$ we apply Taylor's Theorem for $f$ and use inequality $|e^{t}-1|\leq|t|e^{|t|}$
		\begin{align*}
			|I_{11}&(N)-I_{B1}(N)|=\bigg|g(0)\int_{U_{N}}\exp\big(N f(0,N)-N|f'(0,N)|x\big)\bigg[\exp\bigg(\frac{N}{2}f''(x_{\theta},N)x^{2}\bigg)-1\bigg]dx\bigg|\leq\\
			\leq &|g(0)|\int_{U_{N}}\exp\big(N f(0,N)-N|f'(0,N)|x\big)\frac{N}{2}\big|f''(x_{\theta},N)\big|x^{2}\exp\bigg(\frac{N}{2}f''(x_{\theta},N)x^{2}\bigg)dx.
		\end{align*}
		Since the integration domain is $U_{N}$, we have $|x|\leq \frac{1}{N^{1/2}}$. Hence last expression is bounded by
		\begin{align*}
			\label{Extended_Laplace_proof_laplace_approximation_one_dimension_interior_of_domain_approximation_of_first_term_expression_with_constants}
			\leq\frac{1}{2}NGF^{(2)}\exp\bigg(N f(0,N)+\frac{1}{2}F^{(2)}\bigg)\int_{U_{N}}x^{2}e^{-N |f'(0,N)|x}dx,
		\end{align*}
		where $G$ and $F^{(2)}$ are given by (\ref{Laplace_integral_constants_G_G1}) and (\ref{Laplace_integral_boundary_constant_F2}).\\
		Further, for above integral we have
		\begin{align*}
			\int_{U_{N}}&x^{2}e^{-N |f'(0,N)|x}dx\leq\int_{0}^{\infty}x^{2}e^{-N |f'(0,N)|x}dx\leq\frac{2}{\big(N F'^{(1)}\big)^{3}},
		\end{align*}
		where $F'^{(1)}$ is given by (\ref{Laplace_integral_boundary_constant_F'1}).\\
		Hence we have
		\begin{equation*}
			\label{Extended_Laplace_proof_laplace_approximation_one_dimension_interior_of_domain_approximation_of_integral_I11_IG1}
			|I_{11}(N)-I_{B1}(N)|\leq \frac{1}{N^{2}}\frac{GF^{(2)}}{\big(F'^{(1)}\big)^{3}}\exp\bigg(N f(0,N)+\frac{1}{2}F^{(2)}\bigg).
		\end{equation*}	
		For the integral $|I_{12}(N)|$ we also use Taylor's Theorem, i.e. $f(x,N)=f(0,N)+f'(x_{\theta}(N),N)x$. Since $f(0,N)$ has a unique maximum and $U_{N}\subset\Omega'$ we get an estimate
		\begin{equation*}
			f(x,N)\leq f(0,N)-F'^{(1)}x.
		\end{equation*}
		We insert it into $|I_{12}(N)|$ and get an upper bound
		\begin{equation*}
			\label{Extended_Laplace_proof_laplace_approximation_one_dimension_interior_of_domain_approximation_of_integral_I12}
			|I_{12}(N)|\leq e^{Nf(0,N)}G^{(1)}\int_{0}^{\infty}x\exp\big(-NF'^{(1)}x\big)dx\leq G^{(1)}e^{N f(0,N)}\frac{1}{\big(NF'^{(1)}\big)^{2}},
		\end{equation*}
		where $G^{(1)}$ is given by (\ref{Laplace_integral_constants_G_G1}).\\
		For $|I_{2}(N)|$ we also use estimate of $f$ used for $I_{12}$. Then find an upper bound and integrate
		\begin{equation*}
			\label{Extended_Laplace_proof_laplace_approximation_one_dimension_interior_of_domain_approximation_of_integral_I2}
			|I_{2}(N)|\leq\int_{N^{-\frac{1}{2}}}^{\infty}|g(x)|\exp\big(N f(0,N)-N F'^{(1)}x\big)dx=Ge^{N f(0,N)}\frac{1}{NF'^{(1)}}\exp\big(-N^{1/2}F'^{(1)}\big).
		\end{equation*}
		In case of $|I_{3}(N)|$ we have following estimate
		\begin{align*}
			|I_{3}(N)|&\leq e^{Nf(0,N)}\int_{\Omega\backslash\Omega'}|g(x)|\exp\big(N_{0}(f(x,N)-f(0,N))-(N-N_{0})\Delta\big)dx\leq\\
			&\leq e^{(N-N_{0})(f(0,N)-\Delta)}\int_{\Omega\backslash\Omega'}|g(x)|e^{N_{0}f(x,N)}dx\leq Ce^{(N-N_{0})(f(0,N)-\Delta)},
		\end{align*}
		where  $\Delta$ is given by (\ref{Laplace_integral_delta}) and the last inequality is by (\ref{Laplace_integral_bound}).\\
		The integral $|I_{B2}(N)|$ we approximate similarly to $I_{2}$, hence
		\begin{equation*}
			\label{Extended_Laplace_proof_laplace_approximation_one_dimension_interior_of_domain_approximation_of_integral_IG2}
			|I_{B2}(N)|\leq e^{N f(0,N)}\frac{G}{NF'^{(1)}}\exp\big(-N^{1/2}F'^{(1)}\big).
		\end{equation*}
		Now, we combine above approximations and conclude with
		\begin{align*}
			&|I(N)-I_{B}(N)|\leq e^{N f(0,N)}\frac{1}{N^{2}}\Bigg[\frac{GF^{(2)}}{\big(F'^{(1)}\big)^{3}}\exp\bigg(\frac{1}{2}F^{(2)}\bigg)+\frac{G^{(1)}}{\big(F'^{(1)}\big)^{2}}+\frac{2G}{F'^{(1)}}N\exp\big(-N^{1/2}F'^{(1)}\big)+\\
			&+CN^{2}\exp\big(-N\Delta-N_{0}(f(x^{*},N)-\Delta)\big)\Bigg].
		\end{align*}
	\end{proof}


\subsection{Multivariate integral with maximum in the interior}

	\begin{theorem}
		For the integral (\ref{Laplace_integral_interior}) following approximation holds
		\begin{align*}
			\int_{\Omega}g(x)e^{N f(x,N)}dx=e^{N f(x^{*}(N),N)}\bigg(\frac{2\pi}{N}\bigg)^{\frac{m}{2}}\Bigg(\frac{g(x^{*}(N))}{\sqrt{|\det D^{2}f(x^{*}(N),N)|}}+\frac{\omega(N)}{\sqrt{N}}\Bigg),
		\end{align*}
		where $N\geq N_{1}$, for $N_{1}=\max\{\lceil 1/\varepsilon^3\rceil,N_{0}\}$, $\omega(N)=O(1)\ \text{as}\ N\to\infty$ and 
		\begin{align*}
			|\omega&(N)|\leq\frac{\sqrt{2}\Gamma(\frac{m+3}{2})}{\big(F'^{(2)}\big)^{\frac{m+1}{2}}\Gamma(\frac{m}{2})}\bigg[\frac{GF^{(3)}}{3F'^{(2)}}\exp\bigg(\frac{F^{(3)}}{6}\bigg)+\frac{2G^{(1)}}{m+1}\bigg]+\sqrt{N}\bigg(\frac{2\pi}{N}\bigg)^{-\frac{m}{2}}\times\\
			&\times\Bigg[G\exp\big(-N^{\frac{1}{3}}F'^{(2)}\big)\Bigg(|\Omega'|+\frac{\exp(F'^{(2)}/2)}{(2\pi)^{-\frac{m}{2}}F'^{(2)}_{det}}\Bigg)+C\exp\big(-N\Delta-N_{0}(f(x^{*}(N),N)-\Delta)\big)\Bigg].
		\end{align*}
	\end{theorem}

	\begin{proof}
		First we define
		\begin{align*}
			I_{G}(N):&=g(x^{*}(N))e^{N f(x^{*}(N),N)}\bigg(\frac{2\pi}{N}\bigg)^{\frac{m}{2}}\frac{1}{\sqrt{|\det D^{2}f(x^{*}(N),N)|}}=\\
			&=g(x^{*}(N))e^{N f(x^{*}(N),N)}\int_{\mathbb{R}^{m}}\exp\bigg(\frac{1}{2}N(x-x^{*}(N))^{T}D^{2}f(x^{*}(N),N)(x-x^{*}(N))\bigg)dx,
		\end{align*}
		 and introduce a set $U_{N}:=\{x:|x-x^{*}(N)|\leq \frac{1}{N^{1/3}}\}$. For $N\geq N_{1}$ where $N_{1}=\max\{\lceil 1/\varepsilon^{3}\rceil,N_{0}\}$ we have that $U_{N}\subset \Omega'$.\\ Then using Taylor's Theorem we decompose $I(N)$ and $I_{G}(N)$ 
		\begin{align*}
			I(N)&=I_{11}(N)+I_{12}(N)+I_{2}(N)+I_{3}(N):=g(x^{*}(N))\int_{U_{N}}e^{N f(x,N)}dx+\\
			+&\int_{U_{N}}Dg(x_{\theta}(N))^{T}(x-x^{*}(N))e^{N f(x,N)}dx+\int_{\Omega'\backslash U_{N}}g(x)e^{N f(x,N)}dx+\int_{\Omega\backslash\Omega'}g(x)e^{N f(x,N)}dx,\\
			\label{Extended_Laplace_proof_laplace_approximation_m_dimension_interior_of_domain_laplace_first_integral_first_decomposition}
			I_{G}(N&)=I_{G1}(N)+I_{G2}(N):=\\
			=&g(x^{*}(N))e^{N f(x^{*}(N),N)}\int_{U_{N}}\exp\bigg(\frac{1}{2}N(x-x^{*}(N))^{T}D^{2}f(x^{*}(N),N)(x-x^{*}(N))\bigg)dx+\notag\\
			+&g(x^{*}(N))e^{N f(x^{*}(N),N)}\int_{\mathbb{R}^{m}\backslash U_{N}}\exp\bigg(\frac{1}{2}N(x-x^{*}(N))^{T}D^{2}f(x^{*}(N),N)(x-x^{*}(N))\bigg)dx.
		\end{align*} 
		We combine above integrals into	
		\begin{equation*}
			|I(N)-I_{G}(N)|\leq|I_{11}(N)-I_{G1}(N)|+|I_{12}(N)|+|I_{2}(N)|+|I_{3}(N)|+|I_{G2}(N)|.
		\end{equation*}
		For the expression $|I_{11}(N)-I_{G1}(N)|$ we use third order Taylor's Theorem to obtain			
		\begin{align*}
			|I_{11}&(N)-I_{G1}(N)|=\\
			&=|g(x^{*}(N))|\int_{U_{N}}\exp\bigg(N f(x^{*}(N),N)+\frac{1}{2}N(x-x^{*}(N))^{T}D^{2}f(x^{*}(N),N)(x-x^{*}(N))\bigg)\times\\
			&\times\bigg[\exp\bigg(\frac{1}{6}N D^{3}f(x_{\theta}(N),N)(x-x^{*}(N))^{3}\bigg)-1\bigg]dx,
		\end{align*}
		due to $Df(x^{*}(N),N)^{T}(x-x^{*}(N))=0$, since $x^{*}(N)$ is a critical point.\\
		The the third term in the Taylor's Theorem can be bounded
		\begin{equation*}
			\label{Extended_Laplace_proof_taylors_theorem_term_upper_bound}
			|D^{3}f(x_{\theta})x^{m}|\leq\|D^{3}f(x_{\theta})\| |x|^{m}.
		\end{equation*}
		Next, using above inequality and $|e^{t}-1|\leq|t|e^{|t|}$ yields
		\begin{align*}
			|I_{11}&(N)-I_{G1}(N)|\leq\frac{1}{6}NGF^{(3)}\exp\bigg(N f(x^{*}(N),N)+\frac{1}{6}F^{(3)}\bigg)\times\\
			&\times\int_{U_{N}}|x-x^{*}(N)|^{3}\exp\bigg(\frac{1}{2}N(x-x^{*}(N))^{T}D^{2}f(x^{*}(N),N)(x-x^{*}(N))\bigg)dx,\notag
		\end{align*}
		since $|x-x^{*}(N)|\leq \frac{1}{N^{1/3}}$ for any $x\in U_{N}$, where $G$ and $F^{(3)}$ are defined by (\ref{Laplace_integral_constants_G_G1}) and (\ref{Laplace_integral_interior_constant_F3}).\\
		Then by calculus result  $\int_{\mathbb{R}^{m}}|x|^{3}e^{-a|x|^{2}}dx=a^{-\frac{m+3}{2}}\pi^{\frac{m}{2}}\Gamma(\frac{m+3}{2})/\Gamma(\frac{m}{2})$ and the fact that
		\begin{align*} 
			&(x-x^{*}(N))^{T}D^{2}f(x^{*}(N),N)(x-x^{*}(N))=-\big|(x-x^{*}(N))^{T}D^{2}f(x^{*}(N),N)(x-x^{*}(N))\big|\leq\\
			&\leq -|x-x^{*}(N)|^{2}\|D^{2}f(x^{*}(N),N)^{-1/2}\|^{-2},
		\end{align*}
		we obtain an estimate
		\begin{align*}
			&\int_{U_{N}}|x-x^{*}(N)|^{3}\exp\bigg(\frac{1}{2}N(x-x^{*}(N))^{T}D^{2}f(x^{*}(N),N)(x-x^{*}(N))\bigg)dx\leq\\
			&\leq \int_{\mathbb{R}^{m}}|x|^{3}\exp\bigg(-\frac{1}{2}NF'^{(2)}x^{2}\bigg)dx=\pi^{\frac{m}{2}}\frac{\Gamma(\frac{m+3}{2})}{\Gamma(\frac{m}{2})}\bigg(\frac{1}{2}NF'^{(2)}\bigg)^{-\frac{m+3}{2}},
		\end{align*}
		where $F'^{(2)}$ is given by (\ref{Laplace_integral_interior_constant_F'2}). Therefore
		\begin{equation*}
			\label{Extended_Laplace_proof_laplace_approximation_m_dimension_interior_of_domain_approximation_of_integral_I11_IG1}
			|I_{11}(N)-I_{G1}(N)|\leq\frac{1}{6}GF^{(3)}\exp\bigg(\frac{1}{6}F^{(3)}\bigg)\pi^{\frac{m}{2}}\frac{\Gamma(\frac{m+3}{2})}{\Gamma(\frac{m}{2})}\bigg(\frac{F'^{(2)}}{2}\bigg)^{-\frac{m+3}{2}}\frac{1}{\sqrt{N}}e^{Nf(x^{*}(N),N)}N^{-\frac{m}{2}}.
		\end{equation*} 
		Now let us consider integral $I_{12}(N)$. Here again we apply Taylor's Theorem to obtain an upper bound
		\begin{equation}
			\label{Laplace_integral_interior_proof_f_estimate}
			f(x,N)\leq f(x^{*}(N),N)-\frac{1}{2}F'^{(2)}(x-x^{*}(N))^{2},
		\end{equation} 
		and since $g$ has bounded derivative in $U_{N}$ we have
		\begin{equation*}
			|I_{12}(N)|\leq G^{(1)}e^{N f(x^{*}(N),N)}\int_{U_{N}}|x-x^{*}(N)|\exp\bigg(-\frac{1}{2}NF'^{(2)}(x-x^{*}(N))^{2}\bigg)dx.
		\end{equation*}
		Then, using $\int_{\mathbb{R}^{m}}|x|e^{-a|x|^{2}}dx=a^{-\frac{m+1}{2}}\pi^{\frac{m}{2}}\Gamma(\frac{m+1}{2})/\Gamma(\frac{m}{2})$ yields
		\begin{equation*}
			\label{Extended_Laplace_proof_laplace_approximation_m_dimension_interior_of_domain_approximation_of_integral_I2}
			|I_{12}(N)|\leq\pi^{\frac{m}{2}}\frac{\Gamma(\frac{m+1}{2})}{\Gamma(\frac{m}{2})}\bigg(\frac{F'^{(2)}}{2}\bigg)^{-\frac{m+1}{2}}G^{(1)}\frac{1}{\sqrt{N}}e^{N f(x^{*}(N),N)}N^{-\frac{m}{2}}.
		\end{equation*}	
		For $I_{2}$ we also use (\ref{Laplace_integral_interior_proof_f_estimate}) and obtain
		\begin{equation*}
			|I_{2}(N)|\leq e^{N f(x^{*}(N),N)}\int_{\Omega'\backslash U_{N}}g(x)\exp\bigg(-\frac{1}{2}NF'^{(2)}(x-x^{*}(N))^{2}\bigg)dx.
		\end{equation*}
		Since in the set $\Omega'\backslash U_{N}$ function $g$ is bounded by $G$ and $|x-x^{*}(N)|>\frac{1}{N^{1/3}}$, hence
		\begin{align*}
 			|I_{2}(N)|\leq e^{N f(x^{*}(N),N)}G|\Omega'|\exp\bigg(-\frac{1}{2}N^{1/3}F'^{(2)}\bigg).
		\end{align*}
		In case of $|I_{3}(N)|$ we have following upper bound
		\begin{align*}
			|I_{3}(N)|&\leq e^{Nf(x^{*}(N),N)}\int_{\Omega\backslash\Omega'}|g(x)|\exp\big(N_{0}(f(x,N)-f(x^{*},N))-(N-N_{0})\Delta\big)dx\leq\\
			&\leq \exp\big((N-N_{0})(f(x^{*},N)-\Delta)\big)\int_{\Omega\backslash\Omega'}|g(x)|e^{N_{0}f(x,N)}dx\leq\\
			&\leq e^{Nf(x^{*}(N),N)}C\exp\big(-N\Delta-N_{0}(f(x^{*},N)-\Delta)\big),
		\end{align*}
		where the last inequality is due to assumption (\ref{Laplace_integral_bound}).\\
		Integral $I_{G2}$ we estimate similarly to above yielding
		\begin{align*}
 			|I_{G2}(N)|\leq&|g(x^{*}(N))|e^{N f(x^{*}(N),N)}\exp\big(-\frac{1}{2}N^{1/3}F'^{2}+\frac{1}{2}F'^{2}\big)\times\\
			&\times\int_{\mathbb{R}^{m}}\exp\bigg(\frac{1}{2}(x-x^{*}(N))^{T}D^{2}f(x^{*}(N),N)(x-x^{*}(N))\bigg)dx\leq \\
			&\leq Ge^{N f(x^{*}(N),N)}\exp\bigg(-\frac{1}{2}N^{1/3}F'^{(2)}+\frac{1}{2}F'^{(2)}\bigg)\frac{(2\pi)^{\frac{m}{2}}}{F'^{(2)}_{det}}.
		\end{align*}
		where $F'^{(2)}_{det}$ is given by (\ref{Laplace_integral_interior_constant_F'2det}).\\
		Then we combine above approximations 
		\begin{align*}
			\label{Laplaca_approximation_m-dimensional_interior_proof_combined_estimate}
			|I(&N)-I_{G}(N)|\leq\frac{1}{\sqrt{N}}e^{N f(x^{*}(N),N)}\bigg(\frac{2\pi}{N}\bigg)^{\frac{m}{2}}\Bigg[\frac{\sqrt{2}\Gamma(\frac{m+3}{2})}{\Gamma(\frac{m}{2})}\big(F'^{(2)}\big)^{-\frac{m+1}{2}}\times\notag\\
			&\times\bigg(\frac{GF^{(3)}}{3F'^{(2)}}\exp\bigg(\frac{1}{6}F^{(3)}\bigg)+\frac{2G^{(1)}}{m+1}\bigg)+\Bigg((2\pi)^{-\frac{m}{2}}|\Omega'|+\frac{\exp\big(F'^{(2)}/2\big)}{F'^{(2)}_{det}}\Bigg)\times\\
			&\times N^{\frac{m+1}{2}}G\exp\bigg(-\frac{1}{2}N^{1/3}F'^{(2)}\bigg)+\sqrt{N}\bigg(\frac{2\pi}{N}\bigg)^{-\frac{m}{2}}C\exp\big(-N\Delta-N_{0}(f(x^{*},N)-\Delta)\big)\Bigg].
		\end{align*}
	\end{proof}


\subsection{Multivariate integral with maximum on the boundary}
	\begin{theorem}
		For the integral (\ref{Laplace_integral_boundary}) following approximation holds
		\begin{equation*}
			I(N)=e^{N f(x^{*}(N),N)}\frac{1}{N}\bigg(\frac{2\pi}{N}\bigg)^{\frac{m-1}{2}}\Bigg(\frac{g(x^{*}(N))}{\Big|\frac{\partial f(x^{*}(N),N)}{\partial x_{1}}\Big|\sqrt{\big|\det D_{y}^{2}f(x^{*}(N),N)\big|}}+\frac{\omega(N)}{F'^{(1)}\sqrt{N}}\Bigg),
		\end{equation*}
		where $N\geq N_{1}$ for $N_{1}=\max\{\lceil 1/\varepsilon^{3}\rceil,\lceil 1/\varepsilon^{2}\rceil,N_{0}\}$, $\omega(N)=O(1)\ \text{as}\ N\to\infty$ and
		\begin{align*}
			|\omega&(N)|\leq\frac{\sqrt{2}\Gamma(\frac{m+3}{2})}{\big(F'^{(2)}\big)^{\frac{m+1}{2}}\Gamma(\frac{m}{2})}\Bigg(\frac{GF^{(3)}}{3F'^{(2)}}\exp\bigg(\frac{F^{(3)}}{6}\bigg)+\frac{2G^{(1)}}{m+1}\Bigg)+\frac{1}{\sqrt{N}F'^{(2)}_{det}}\bigg(\frac{G^{(1)}}{F'^{(1)}}+\frac{m^{2}GF^{(3)}}{2F'^{(1)}F'^{(2)}}\bigg)+\\
			&+\frac{1}{\sqrt{N}}\bigg(\frac{G}{F'^{(2)}_{det}}+\frac{1}{\sqrt{N}}\bigg)\Bigg[\frac{F^{(2)}}{\big(F'^{(1)}\big)^{2}}\exp\bigg(\frac{F^{(2)}}{2}\bigg)+2\exp\big(-N^{\frac{1}{2}}F'^{(1)}\big)\Bigg]+\sqrt{N}\bigg(\frac{2\pi}{N}\bigg)^{-\frac{m}{2}}\times\\
			&\times\Bigg[G\exp\big(-N^{\frac{1}{3}}F'^{(2)}\big)\bigg(|\Omega'|+\frac{\exp\big(F^{(2)}/2)}{(2\pi)^{-\frac{m}{2}}F'^{(2)}_{det}}\bigg)+C\exp\big(-N\Delta-N_{0}(f(x^{*},N)-\Delta)\big)\Bigg].
		\end{align*}
	\end{theorem}

	\begin{proof} 
		We decompose $I(N)$ into
		\begin{equation*}
			I(N)=I_{1}(N)+I_{2}(N):=\int_{\Omega'\cap\{x:x_{1}\geq 0\}}g(x)e^{N f(x,N)}dx+\int_{\Omega\cap\{x:x_{1}\geq 0\}\backslash\Omega'}g(x)e^{N f(x,N)}dx.\\
		\end{equation*}
		We approximate $I_{2}(N)$ as the integral $I_{3}$ in the previous proof. Since $\Omega\backslash\Omega'\subset\Omega\backslash U_{N}$
		\begin{align*}
			|I_{2}(N)|&\leq\exp\big((N-N_{0})(f(x^{*}(N),N)-\Delta)\big)\int_{\Omega\cap\{x:x_{1}\geq 0\}\backslash\Omega'}|g(x)|e^{N_{0}f(x,N)}dx\leq\\
			&\leq C\exp\big((N-N_{0})(f(x^{*}(N),N)-\Delta)\big).
		\end{align*}
		Then we express integral $I_{1}(N)$ as
		\begin{equation*}
			I_{1}(N)=\int_{0}^{\varepsilon}I_{1}(x_{1},N)dx_{1},
		\end{equation*}
		and
		\begin{equation*}
			I_{1}(x_{1},N):=\int_{\Omega'(x_{1})}g(x_{1},y)e^{Nf(x_{1},y,N)}dy,
		\end{equation*}
		where $\Omega'(x_{1})=\{y:(x_{1},y)\in\Omega'\}$.\\
		Next, we apply Theorem 2
		\begin{align*}
			I_{1}(x_{1},N)=e^{N f(x_{1},y^{*}(x_{1},N),N)}\bigg(\frac{2\pi}{N}\bigg)^{\frac{m-1}{2}}\Bigg(\frac{g(x_{1},y^{*}(x_{1},N))}{\sqrt{|\det D_{y}^{2}f(x_{1},y^{*}(x_{1},N),N)|}}+\frac{\omega_{I}(x_{1},N)}{\sqrt{N}}\Bigg),
		\end{align*}
		where $y^{*}(x_{1},N)=\argmax_{y\in\Omega'(x_{1})}f(x_{1},y,N)$.\\
		Due to integration over the set $\Omega'(x_{1})$ the constants which occurs as a result of application of Theorem 2 can be replaced by the appropriate constants for a larger set $\Omega'$ which are independent of $x_{1}$, that is (\ref{Laplace_integral_boundary_constant_F'2}), (\ref{Laplace_integral_boundary_constant_F'2det}),(\ref{Laplace_integral_boundary_constant_F2}) and (\ref{Laplace_integral_boundary_constant_F3}).\\
		Then we apply Theorem 1 to $I_{1}(N)$
		\begin{align*}
			I_{1}(N)=&e^{N f(x^{*}(N),N)}\frac{1}{N}\bigg(\frac{2\pi}{N}\bigg)^{\frac{m-1}{2}}\Bigg(\frac{g(x^{*}(N))}{\big|\frac{\partial f(x^{*}(N),N)}{\partial x_{1}}\big|\sqrt{|\det D_{y}^{2}f(x^{*}(N),N)|}}+\\
			&+\frac{\omega_{B1}(N)}{N\sqrt{\inf_{x_{1}\in(0,\varepsilon)}|\det D_{y}^{2}f(x_{1},y^{*}(x_{1},N),N)|}}+\frac{\omega_{I}(N)}{\big|\frac{\partial f(x^{*}(N),N)}{\partial x_{1}}\big|\sqrt{N}}+\frac{\omega_{B2}(N)}{N^{\frac{2}{3}}}\Bigg),
		\end{align*}
		where  $(0,y^{*}(0,N))=x^{*}(N)$.\\
		Since Theorem 1 was applied where the integration domain was the curve $y^{*}(x_{1},N)$, the constants in the estimate of $\omega_{B1}(N)$ and $\omega_{B2}(N)$ can be replaced by the constants for larger set $\Omega'$ i.e. (\ref{Laplace_integral_boundary_constant_F'1}), (\ref{Laplace_integral_boundary_constant_F'2}) and (\ref{Laplace_integral_boundary_constant_F2}). 
		\begin{align*}
			|\omega_{B1}(N)|\leq &\frac{GF^{(2)}}{\big(F'^{(1)}\big)^{3}}\exp\bigg(\frac{1}{2}F^{(2)}\bigg)+\frac{1}{\big(F'^{(1)}\big)^{2}}\bigg(G^{(1)}+\frac{m^{2}GF^{(3)}}{2F'^{(2)}}\bigg)+\\
			&+\frac{2G}{F'^{(1)}}N\exp\big(-N^{1/2}F'^{(1)}\big),\\
			|\omega_{B2}(N)|\leq &\frac{F^{(2)}}{\big(F'^{(1)}\big)^{3}}\exp\bigg(\frac{1}{2}F^{(2)}\bigg)+\frac{2}{F'^{(1)}}N\exp\big(-N^{1/2}F'^{(1)}\big).
		\end{align*}
		Then we combine above results with the estimate of $I_{2}(N)$ to obtain the final result.
	\end{proof}


\section{Limit theorems}

	For the integral (\ref{Laplace_integral_interior}) and (\ref{Laplace_integral_boundary}) let us assume
	\begin{equation}
		\label{probability_equation_for_f}
		f(x,N)=f(x)+\epsilon(N)\sigma(x),
	\end{equation}
	where $\sigma, f$ are some functions with derivatives up to second order at $x^{*}$, Hessian $D^{2}f(x^{*})$ is nonsingular and $\epsilon(N)>0, \epsilon(N)\to 0$ as $N\to\infty$.\\
	Let us consider a first order Taylor expansion of derivatives vector $Df(x^{*}(N),N)$ at $x^{*}$
	\begin{equation}
		\label{probability_estimate_maximums_derivation}
		Df(x^{*}(N),N) = Df(x^{*},N) + D^{2}f(x_{\theta}(N),N)(x^{*}(N)-x^{*}).
	\end{equation}
	For the integral (\ref{Laplace_integral_interior}), $x^{*}(N)$ and $x^{*}$ are a critical points. Hence $Df(x^{*}(N),N) = Df(x^{*})=0$ and due to (\ref{probability_equation_for_f}), $Df(x^{*},N) = \epsilon(N)D\sigma(x^{*})$. As a result we obtain 
	\begin{equation}
		\label{probability_estimate_maximums}
		x^{*}(N)=x^{*}+\epsilon(N)O(1),\ N\to\infty.
	\end{equation} 
	Equation (\ref{probability_estimate_maximums}) is also valid for the integral (\ref{Laplace_integral_boundary}), as for this situation $x_{1}^{*}(N)=x_{1}^{*}$.\\
	For every $N\geq N_{0}$, let $X(N)$ be a random vector with distributions defined using integral (\ref{Laplace_integral_interior}) and (\ref{Laplace_integral_boundary})
	\begin{align}
		\label{probability_distribution_interior}
		(a)&\ P_{N}(X(N)\in A)=\frac{\int_{A}e^{Nf(x,N)}dx}{\int_{\Omega}e^{Nf(x,N)}dx},\\
		\label{probability_distribution_boundary}
		(b)&\ P_{N}(X(N)\in B)=\frac{\int_{B}e^{Nf(x,N)}dx}{\int_{\Omega\cap\{x:x_{1}\geq 0\}}e^{Nf(x,N)}dx},
	\end{align}
	where $A\in\mathcal{B}(\Omega)$ and $B\in\mathcal{B}(\Omega\cap\{x:x_{1}\geq 0\})$.


\subsection{Weak law of large numbers}

	\begin{theorem}[Weak law of large numbers]
		As $N\to \infty$ the random vector $X(N)$ converges in distribution to a constant $x^{*}$ and following estimate of the mgf holds
		\begin{equation*}
			M_{X(N)}(\xi)=e^{\xi^{T}x^{*}}\bigg(1+\frac{O(1)}{\sqrt{N}}+O(1)\epsilon(N)\bigg),\ N\to\infty.
		\end{equation*}
	\end{theorem}
	
	\begin{remark}
		For this and two following limit theorems estimate of the convergence error term can be explicitly estimated with use of the previous results.
	\end{remark}
	
	\begin{proof}
		To prove the convergence of $X(N)$ it is sufficient to prove convergence of its moment generating functions
		\begin{align*}
			(a)&\ M_{X(N)}(\xi)=\frac{\int_{\Omega}e^{\xi^{T}x}e^{Nf(x,N)}dx}{\int_{\Omega}e^{Nf(x,N)}dx},\\
			(b)&\ M_{X(N)}(\xi)=\frac{\int_{\Omega\cap\{x:x_{1}\geq 0\}}e^{\xi^{T}x}e^{Nf(x,N)}dx}{\int_{\Omega\cap\{x:x_{1}\geq 0\}}e^{Nf(x,N)}dx},
		\end{align*}
		where $|\xi|<h$, for some $h>0$.\\
		We approximate the denominator with use of Theorem 2 
		\begin{align*}
			&(a)\int_{\Omega}e^{Nf(x,N)}dx=e^{N f(x^{*}(N),N)}\bigg(\frac{2\pi}{N}\bigg)^{\frac{m}{2}}\frac{1}{\sqrt{|\det D^{2}f(x^{*}(N),N)|}}\bigg(1+\frac{O(1)}{\sqrt{N}}\bigg),\\
			&(b)\int_{\Omega\cap\{x:x_{1}\geq 0\}}e^{Nf(x,N)}dx=e^{N f(x^{*}(N),N)}\frac{1}{N}\bigg(\frac{2\pi}{N}\bigg)^{\frac{m-1}{2}}\frac{1}{\Big|\frac{\partial f(x^{*}(N),N)}{\partial x_{1}}\Big|\sqrt{|\det D_{y}^{2}f(x^{*}(N),N)|}}\times\\
			&\qquad\times\bigg(1+\frac{O(1)}{\sqrt{N}}\bigg),
		\end{align*}	
		and numerator
		\begin{align*}
			(a)\int_{\Omega}\exp\big(\xi^{T}&x+Nf(x,N)\big)dx=\exp\big(\xi^{T}x^{*}(N)+N f(x^{*}(N),N)\big)\bigg(\frac{2\pi}{N}\bigg)^{\frac{m}{2}}\times\\
			&\times\frac{1}{\sqrt{|\det D^{2}f(x^{*}(N),N)|}}\bigg(1+\frac{O(1)}{\sqrt{N}}\bigg),\\
			(b)\int_{\Omega\cap\{x:x_{1}\geq 0\}}&\exp\big(\xi^{T}x+Nf(x,N)\big)dx=\exp\big(\xi^{T}x^{*}(N)+N f(x^{*}(N),N)\big)\times\\
			&\times\frac{1}{N}\bigg(\frac{2\pi}{N}\bigg)^{\frac{m-1}{2}}\frac{1}{\big|\frac{\partial f(x^{*}(N),N)}{\partial x_{1}}\big|\sqrt{|\det D_{y}^{2}f(x^{*}(N),N)|}}\bigg(1+\frac{O(1)}{\sqrt{N}}\bigg).
		\end{align*}
		Dividing the approximations of denominators and numerators yields
		\begin{equation*}
			\label{limit_proof_limit_mgf}
			 M_{X(N)}(\xi)=\exp\big(\xi^{T}x^{*}(N)\big)\bigg(1+\frac{O(1)}{\sqrt{N}}\bigg).
		\end{equation*}
 		Next, we use estimate (\ref{probability_estimate_maximums})  and Taylor's expansion for the exponent function to obtain
		\begin{equation*}
			 M_{X(N)}(\xi)=\exp\big(\xi^{T}x^{*}\big)\big(1+\epsilon(N)\xi^{T}O(1)\big)\bigg(1+\frac{O(1)}{\sqrt{N}}\bigg),\notag
		\end{equation*}
		which yields the result of the theorem.
	\end{proof}


\subsection{Central limit theorem}

	Here, let us consider $X(N)$ with bounded range set $\Omega$ and for (\ref{probability_equation_for_f}) let $\epsilon(N)=o\big(\frac{1}{\sqrt{N}}\big)$. Then we have following preliminary result
	\begin{proposition}
		For the function $\widetilde{f}(x,N):=f(x,N)+\frac{1}{\sqrt{N}}\xi^{T}(x^{*}-x)$ with $\xi>0$ following approximations hold
		\begin{align}
			\label{probability_tilda_estimates_maximums}
			&\widetilde{x}^{*}(N)-x^{*}(N)=D^{2}f(x^{*})^{-1}\frac{\xi}{\sqrt{N}}+\frac{O(1)}{N},\\
			\label{probability_tilda_estimates_functions}
			&\widetilde{f}(\widetilde{x}^{*}(N),N)-f(x^{*}(N),N)=\frac{1}{2N}\xi^{T}D^{2}f(x^{*})^{-1}\xi+\frac{O(1)}{N^{3/2}}+\frac{O(1)\epsilon(N)}{\sqrt{N}},\\
			\label{probability_tilda_estimates_determinants}
			&\frac{\sqrt{|\det D^{2}f(x^{*}(N),N)|}}{\sqrt{|\det D^{2}\widetilde{f}(\widetilde{x}^{*}(N),N)|}}=1+\frac{O(1)}{\sqrt{N}},
		\end{align}
		as $N\to\infty$, where $\widetilde{x}^{*}(N)$ is a maximum of $\widetilde{f}$.
	\end{proposition}
	\begin{proof}
		We begin with the proof of (\ref{probability_tilda_estimates_maximums}).\\
		Note that for large enough $N$ the properties of the function $\widetilde{f}$ are the same as of $f$ but with the maximum at $\widetilde{x}^{*}(N)$.\\
		We use first order Taylor's expansion for $D\widetilde{f}(\widetilde{x}^{*}(N),N)$ at $x^{*}(N)$. Since $\widetilde{x}^{*}(N)$ and $x^{*}(N)$ are critical points we have
		\begin{equation*}
			D^{2}f(x_{\theta}(N),N)(\widetilde{x}^{*}(N)-x^{*}(N))=\frac{1}{\sqrt{N}}\xi.
		\end{equation*}
		Since for large enough $N$, $x_{\theta}(N)\in\Omega'$
		\begin{equation*}
			\big|\widetilde{x}^{*}(N)-x^{*}(N)\big|\leq\frac{|\xi|}{F'^{(2)}\sqrt{N}}.
		\end{equation*}
		Further, we use second order Taylor's expansion for $D\widetilde{f}(\widetilde{x}^{*}(N),N)$ at $x^{*}(N)$ and get an upper bound
		\begin{equation*}
			\bigg|D^{2}f(x^{*}(N),N)(\widetilde{x}^{*}(N)-x^{*}(N))-\frac{1}{\sqrt{N}}\xi\bigg|\leq \frac{F^{(3)}}{2}\big|\widetilde{x}^{*}(N)-x^{*}(N)\big|^{2}.
		\end{equation*}
		Substituting previous estimate yields
		\begin{equation*}
			\label{sum_tilda_estimates_proof_maximums_estimate_inequality}
			\bigg|(\widetilde{x}^{*}(N)-x^{*}(N))-D^{2}f(x^{*}(N),N)^{-1}\frac{1}{\sqrt{N}}\xi\bigg|\leq\frac{F^{(3)}|\xi|^{2}}{(F'^{(2)})^{3}2N}.
		\end{equation*}
		Next we use Taylor's Theorem for $D^{2}f(x^{*}(N),N)$ at $x^{*}$ to obtain
		\begin{equation}
			\label{probability_tilda_estimates_proof_f_second_derivative_estimate}
			D^{2}f(x^{*}(N),N)=D^{2}f(x^{*},N)+ \epsilon(N)O(1)=D^{2}f(x^{*})+ \epsilon(N)O(1),
		\end{equation}
		where the second equality is by (\ref{probability_equation_for_f}).\\
		Due to equation $(I+\epsilon A)^{-1}=I+O(\epsilon), \ \epsilon\to 0$ we get
		\begin{equation*}
			D^{2}f(x^{*},N)^{-1}=D^{2}f(x^{*})^{-1}\bigg[I+\epsilon(N)O(1)D^{2}f(x^{*})^{-1}\bigg]^{-1}=D^{2}f(x^{*})^{-1}+\epsilon(N)O(1).
		\end{equation*}
		Putting together above estimates yields
		\begin{equation*}
			\bigg|(\widetilde{x}^{*}(N)-x^{*}(N))-\frac{1}{\sqrt{N}}D^{2}f(x^{*})^{-1}\xi\bigg|\leq\frac{F^{(3)}|\xi|^{2}}{2(F'^{(2)})^{3}N}+\frac{\epsilon(N)O(1)|\xi|}{\sqrt{N}},
		\end{equation*}
		and as $\epsilon(N)=o\big(\frac{1}{\sqrt{N}}\big)$ consequently
		\begin{equation*}
			\bigg|(\widetilde{x}^{*}(N)-x^{*}(N))-\frac{1}{\sqrt{N}}D^{2}f(x^{*})^{-1}\xi\bigg|\leq\frac{K_{1}}{N},
		\end{equation*}
		for some positive $K_{1}$. Hence we get the first result of the Proposition.\\
		We start the proof of (\ref{probability_tilda_estimates_functions}) with expressing $\widetilde{f}(\widetilde{x}^{*}(N),N)$ in the following form
		\begin{align*}
			\widetilde{f}(\widetilde{x}^{*}(N),N)&=f\bigg(x^{*}(N)+D^{2}f(x^{*})^{-1}\frac{\xi}{\sqrt{N}}+\frac{O(1)}{N},N\bigg)-\\
			&-\frac{1}{\sqrt{N}}\xi^{T}\bigg(x^{*}(N)+D^{2}f(x^{*})^{-1}\frac{\xi}{\sqrt{N}}+\frac{O(1)}{N}-x^{*}\bigg),
		\end{align*}
		and expand it using 3-rd order Taylor's Theorem
		\begin{align*}
			&\widetilde{f}(\widetilde{x}^{*}(N),N)=f(x^{*}(N),N)-\frac{1}{\sqrt{N}}\xi^{T}\bigg(x^{*}(N)+D^{2}f(x^{*})^{-1}\frac{\xi}{\sqrt{N}}+\frac{O(1)}{N}-x^{*}\bigg)+\\
			&+\frac{1}{2}\bigg(D^{2}f(x^{*})^{-1}\frac{\xi}{\sqrt{N}}+\frac{O(1)}{N}\bigg)^{T}D^{2}f(x^{*}(N),N)\bigg(D^{2}(x^{*})^{-1}\frac{\xi}{\sqrt{N}}+\frac{O(1)}{N}\bigg)+\\
			&+\frac{1}{6}D^{3}f(x_{\theta}(N),N)\bigg(D^{2}f(x^{*})^{-1}\frac{\xi}{\sqrt{N}}+\frac{O(1)}{N}\bigg)^{3},
		\end{align*}
		then calculate its upper bound with use of (\ref{probability_tilda_estimates_proof_f_second_derivative_estimate})
		\begin{align*}
			\label{sum_tilda_estimates_proof_functions_third_derivative_bounded}
			&\bigg|\widetilde{f}(\widetilde{x}^{*}(N),N)-f(x^{*}(N),N)+\frac{1}{\sqrt{N}}\xi^{T}\bigg(D^{2}f(x^{*})^{-1}\frac{\xi}{\sqrt{N}}+\frac{O(1)}{N}+\epsilon(N)O(1)\bigg)-\\
			&-\frac{1}{2}\bigg(D^{2}f(x^{*})^{-1}\frac{\xi}{\sqrt{N}}+\frac{O(1)}{N}\bigg)^{T}\big(D^{2}f(x^{*})+\epsilon(N)O(1)\big)\bigg(D^{2}f(x^{*})^{-1}\frac{\xi}{\sqrt{N}}+\frac{O(1)}{N}\bigg)\bigg|\leq\\
			&\leq\frac{1}{6}F^{(3)}\bigg|D^{2}f(x^{*})^{-1}\frac{\xi}{\sqrt{N}}+\frac{O(1)}{N}\bigg|^{3}.\notag
		\end{align*}
		Next we apply Triangle and Schwarz inequalities and conclude with
		\begin{align*}
			&\bigg|\widetilde{f}(\widetilde{x}^{*}(N),N)-f(x^{*}(N),N)+\frac{1}{2N}\xi^{T}D^{2}f(x^{*})^{-1}\xi\bigg|\leq\frac{\epsilon(N)O(1)|\xi|}{\sqrt{N}}+\\
			&+\frac{O(1)}{N^{3/2}}\Bigg[\frac{1}{6}F^{(3)}\bigg|D^{2}f(x^{*})^{-1}\xi+\frac{O(1)}{\sqrt{N}}\bigg|^{3}+\frac{1}{2}\epsilon(N)\sqrt{N}\bigg|D^{2}f(x^{*})^{-1}\xi+\frac{O(1)}{\sqrt{N}}\bigg|^{2}+\\
			&+\frac{1}{2\sqrt{N}}\big\|D^{2}f(x^{*})\big\|\Bigg].
		\end{align*}
		Hence we get
		\begin{equation*}
			\bigg|\widetilde{f}(\widetilde{x}^{*}(N),N)-f(x^{*}(N),N)+\frac{1}{2N}\xi^{T}D^{2}f(x^{*})^{-1}\xi\bigg|\leq\frac{K_{21}}{N^{3/2}}+\frac{K_{22}\epsilon(N)}{\sqrt{N}},
		\end{equation*}
		where $K_{21},K_{22}$ are some positive constants. This completes the proof of (\ref{probability_tilda_estimates_functions}).\\
		For the proof of (\ref{probability_tilda_estimates_determinants}) we apply Taylor's approximation to $D^{2}\widetilde{f}(\widetilde{x}^{*}(N),N)$ at $x^{*}$ and use (\ref{probability_estimate_maximums})
		\begin{equation*}
			\big\|D^{2}\widetilde{f}(\widetilde{x}^{*}(N),N)-D^{2}f(x^{*},N)\big\|\leq \frac{F^{(3)}}{\sqrt{N}}\bigg\|D^{2}f(x^{*})^{-1}\xi+ \frac{O(1)}{\sqrt{N}}\bigg\|.
		\end{equation*} 
		Then, by (\ref{probability_equation_for_f}) we have
		\begin{equation*}
			D^{2}\widetilde{f}(\widetilde{x}^{*}(N),N)=D^{2}f(x^{*})+\frac{O(1)}{\sqrt{N}}.
		\end{equation*} 
		Next, we divide (\ref{probability_tilda_estimates_proof_f_second_derivative_estimate}) by the above estimate and obtain
		 \begin{equation*}
			\label{sum_tilda_estimates_proof_determinants_determinants_approximation}
			\frac{D^{2}f(x^{*}(N),N)}{D^{2}\widetilde{f}(\widetilde{x}^{*}(N),N)}=I+\frac{O(1)}{\sqrt{N}}.
		\end{equation*}
		Now, let us consider Taylor's expansion of $f(x)=\sqrt{x}$ at $1$
		\begin{equation*}
			\sqrt{x}=1+\frac{1}{2\sqrt{x_{\theta}}}(x-1),
		\end{equation*}
		where $x_{\theta}\in(\min(x,1),\max(x,1))$. We apply it for 
		\begin{equation*}
			x=\bigg|\det\bigg(\frac{ D^{2}f(x^{*}(N),N)}{D^{2}\widetilde{f}(\widetilde{x}^{*}(N),N)}\bigg)\bigg|,
		\end{equation*} 
		which yields
		\begin{equation*}
			\label{estimates_tilda_estimates_proof_determinants_inequality_LHS_sqare_root_taylor}
			\Bigg|\frac{\sqrt{|\det D^{2}f(x^{*}(N),N)|}}{\sqrt{|\det D^{2}\widetilde{f}(\widetilde{x}^{*}(N),N)|}}-1\Bigg|=\frac{1}{2\sqrt{x_{\theta}}}\Bigg|\bigg|\det\bigg(\frac{ D^{2}f(x^{*}(N),N)}{D^{2}\widetilde{f}(\widetilde{x}^{*}(N),N)}\bigg)\bigg|-1\Bigg|.
		\end{equation*}
		Since we have defined constants $F^{(2)}$, $F'^{(2)}_{det}$, determinants are bounded, therefore $x_{\theta}$ is also bounded.\\
		Then we can use formula $\det(I+\epsilon A)=I+O(\epsilon),\ \epsilon\to 0$ to get
		\begin{equation*}
			\Bigg|\det\bigg(\frac{D^{2}f(x^{*}(N),N)}{D^{2}\widetilde{f}(\widetilde{x}^{*}(N),N)}\bigg)-1\Bigg|=\frac{1}{2\sqrt{x_{\theta}}}\Bigg|\det\bigg(I+\frac{O(1)}{\sqrt{N}}\bigg)-1\Bigg|\leq\frac{K_{3}}{\sqrt{N}},
		\end{equation*}
		where $K_{3}>0$. Combining above results finalizes the proof of (\ref{probability_tilda_estimates_determinants})
	\end{proof}

	Now, with use of above result we can prove following limit theorems
 
	\begin{theorem}[Central limit theorem I]
		For $X(N)$ with distribution (\ref{probability_distribution_interior}) the random vector $Z(N)=\sqrt{N}(x^{*}-X(N))$ converges weakly to $\mathcal{N}(0,D^{2}f(x^{*})^{-1})$ and following estimate of the mgf holds
		\begin{equation*}
			M_{Z(N)}(\xi)=\exp\bigg(\frac{1}{2}\xi^{T}D^{2}f(x^{*})^{-1}\xi\bigg)\bigg(1+\frac{O(1)}{\sqrt{N}}+\epsilon(N)\sqrt{N}O(1)\bigg),\ N\to\infty.
		\end{equation*}
	\end{theorem}

	\begin{theorem}[Central limit theorem II]
		For $X(N)$ with distribution (\ref{probability_distribution_boundary}) the random vector $Z(N)=\big(N(x_{1}^{*}-X_{1}(N)),\sqrt{N}(y^{*}-Y(N))\big)$ converges weakly to $Exp\big|\frac{\partial f(x^{*})}{\partial x_{1}}\big|$ for $Z_{1}(N)$ and to $\mathcal{N}(0,D_{y}^{2}f(x^{*})^{-1})$ for $\big(Z_{2}(N),\ldots,Z_{m}(N)\big)$. Furthermore, following estimate of the mgf holds
		\begin{equation*}
			M_{Z(N)}(\xi)=\frac{\big|\frac{\partial f(x^{*})}{\partial x_{1}}\big|}{\big|\frac{\partial f(x^{*})}{\partial x_{1}}-\xi_{1}\big|}e^{\frac{1}{2}\hat{\xi}^{T}D_{y}^{2}f(x^{*})^{-1}\hat{\xi}}\bigg(1+\frac{O(1)}{\sqrt{N}}+\epsilon(N)\sqrt{N}O(1)\bigg),\ N\to\infty,
		\end{equation*}
		where $\hat{\xi}=(\xi_{2},\ldots,\xi_{m})$, $X(N)=\big(X_{1}(N),Y(N)\big)$ and $x^{*}=(x_{1}^{*},y^{*})$.
	\end{theorem}

	\begin{proof}[Proof of Theorem 5]
		Let us define  $\widetilde{f}(x,N):=f(x,N)+\frac{1}{N}\xi^{T}(x^{*}-x)$, then the mgf of $Z(N)$ can be expressed
		\begin{equation*}
			\label{fluctuation_proof_mgf}
			M_{Z(N)}(\xi)=\frac{\int_{\Omega}e^{N\tilde{f}(x,N)}dx}{\int_{\Omega}e^{Nf(x,N)}dx}.
		\end{equation*}
		First, we approximate the numerator and denominator of the mgf using Theorem 2
		\begin{equation*}
			M_{Z(N)}(\xi)=\exp\big(N\widetilde{f}(\widetilde{x}^{*}(N),N)-Nf(x^{*}(N),N)\big)\frac{\sqrt{|\det D^{2}f(x^{*}(N),N)|}}{\sqrt{|\det D^{2}\widetilde{f}(\widetilde{x}^{*}(N),N)|}}\bigg(1+\frac{O(1)}{\sqrt{N}}\bigg).
		\end{equation*}
		Next, we insert the estimates (\ref{probability_tilda_estimates_functions}) and (\ref{probability_tilda_estimates_determinants}) from Proposition 1 
		\begin{equation*}
			M_{Z(N)}(\xi)=\exp\bigg(\frac{1}{2}\xi^{T}D^{2}f(x^{*})^{-1}\xi+N^{-1/2}O(1)+\epsilon(N)\sqrt{N}O(1)\bigg)\bigg(1+\frac{O(1)}{\sqrt{N}}\bigg).
		\end{equation*}
		Then use estimate
		\begin{equation}
			\exp\big(N^{-1/2}O(1)+\epsilon(N)\sqrt{N}O(1)\big)=1+N^{-1/2}O(1)+\epsilon(N)\sqrt{N}O(1),
		\end{equation}
		which leads to the final result.
	\end{proof}

	\begin{proof}[Proof of Theorem 6]
		This proof is analogical. Here we define $\widetilde{f}(x,N):=f(x,N)+\frac{1}{\sqrt{N}}\xi^{T}(\sqrt{N}(x^{*}_{1}-x_{1}),x_{2}^{*}-x_{2},\ldots,x^{*}_{m}-x_{m})$. Then approximate the numerator and denominator of the mgf using Theorem 3 
		\begin{align*}
			M_{Z(N)}(\xi)&=\exp\big(N\widetilde{f}(\widetilde{x}^{*}(N),N)-Nf(x^{*}(N),N)\big)\frac{\big|\frac{\partial f(x^{*}(N),N)}{\partial x_{1}}\big|}{\big|\frac{\partial\widetilde{f}(\widetilde{x}^{*}(N),N)}{\partial x_{1}}\big|}\frac{\sqrt{|\det D_{y}^{2}f(x^{*}(N),N)|}}{\sqrt{|\det D_{y}^{2}\widetilde{f}(\widetilde{x}^{*}(N),N)|}}\times\\
			&\times\bigg(1+\frac{O(1)}{\sqrt{N}}\bigg).
		\end{align*}
		Since the first coordinate of $x^{*}(N)$ is fixed we can apply results of Proposition 1 to obtain
		\begin{equation*}
			\label{fluctuations_proof_mgf_estimates_middle_step}
			M_{Z(N)}(\xi)=\frac{\big|\frac{\partial f(x^{*}(N),N)}{\partial x_{1}}\big|}{\big|\frac{\partial f(\widetilde{x}^{*}(N),N)}{\partial x_{1}}-\xi_{1}\big|}\exp\bigg(\frac{1}{2}\hat{\xi}^{T}D^{2}f(x^{*})^{-1}\hat{\xi}\bigg)\bigg(1+\frac{O(1)}{\sqrt{N}}+\epsilon(N)\sqrt{N}O(1)\bigg),
		\end{equation*}
		where $\hat{\xi}=(\xi_{2},\ldots,\xi_{m})$.\\
		Next, we use first order Taylor's Theorem to obtain
		\begin{equation*}
			\bigg|\frac{\partial f(\widetilde{x}^{*}(N),N)}{\partial x_{1}}-\frac{\partial f(x^{*},N)}{\partial x_{1}}\bigg|\leq F^{(2)}\big(|\widetilde{x}^{*}(N)-x^{*}(N)|+|x^{*}(N)-x^{*}|\big),
		\end{equation*}
		and with use of  (\ref{probability_equation_for_f}), (\ref{probability_estimate_maximums}) and (\ref{probability_tilda_estimates_maximums}) get
		\begin{equation*}
			\label{probability_fluctuation_proof_b_case_first_derivative_estimate}
			\frac{\partial f(\widetilde{x}^{*}(N),N)}{\partial x_{1}}=\frac{\partial f(x^{*})}{\partial x_{1}}+\frac{O(1)}{\sqrt{N}}.
		\end{equation*}
		Analogical estimation procedure is for other derivative
		\begin{equation*}
			\label{probability_fluctuation_proof_b_case_first_derivative_estimate}
			\frac{\partial f(x^{*}(N),N)}{\partial x_{1}}=\frac{\partial f(x^{*})}{\partial x_{1}}+O(1)\epsilon(N).
		\end{equation*}
		Using above estimates we obtain
		\begin{equation*}
			\frac{\big|\frac{\partial f(x^{*}(N),N)}{\partial x_{1}}\big|}{\big|\frac{\partial f(\widetilde{x}^{*}(N),N)}{\partial x_{1}}-\xi_{1}\big|}=\frac{\big|\frac{\partial f(x^{*})}{\partial x_{1}}\big|}{\big|\frac{\partial f(x^{*})}{\partial x_{1}}+O(1)/\sqrt{N}-\xi_{1}\big|}(1+\epsilon(N)O(1))=\frac{\big|\frac{\partial f(x^{*})}{\partial x_{1}}\big|}{\big|\frac{\partial f(x^{*})}{\partial x_{1}}-\xi_{1}\big|}\bigg(1+\frac{O(1)}{\sqrt{N}}\bigg),
		\end{equation*}
		and substituting that into last estimate of mgf yields the final result.
	\end{proof}
	
\begin{acknowledgments}
	Author dedicates special thanks to Professor Vassili Kolokoltsov from the University of Warwick for
	\begin{itemize}
		\item extending result of \cite{Laplace_method_approach_kolokoltsov} to a draft for this article, however, excluding the proof for the main results, that is, the remainder estimates and fluctuation theorem for the boundary case,
		\item hint in the proof of Theorem 3 that the curvature of $y^{*}(x_{1},N)$ does not influence the solution,
		\item English language check and style suggestions, 
		\item comments on the applications of the paper results.
	\end{itemize}
\end{acknowledgments}


\bibliographystyle{apa}
\bibliography{biblio}
	
\end{document}